\documentclass{amsart}
\usepackage{xcolor}
\usepackage{amsmath}
\usepackage{amssymb}
\usepackage{amsthm}
\usepackage[shortlabels]{enumitem}
\usepackage{mathtools}
\usepackage{mathrsfs}
\usepackage{multirow}
\usepackage{subfiles}
\usepackage{stmaryrd}
\usepackage[all]{xy}

\newtheorem{theorem}{Theorem}[section]
\newtheorem{lemma}[theorem]{Lemma}
\newtheorem{prop}[theorem]{Proposition}
\newtheorem{corollary}[theorem]{Corollary}
\theoremstyle{definition}
\newtheorem{definition}[theorem]{Definition}

\theoremstyle{remark}

\newcommand{\diag}{\operatorname{diag}}

\newcommand{\Aut}{\operatorname{Aut}}

\newcommand{\Out}{\operatorname{Out}}

\newcommand{\soc}{\operatorname{soc}}

\newcommand{\Stab}{\operatorname{Stab}}

\newcommand{\mleq}{<_{\max}}

\newcommand{\supp}{\operatorname{supp}}

\title{Finite groups with the minimal generating set exchange property}

\author{Andrea Lucchini}
\address{Andrea Lucchini. University of Padova (Italy), Dipartimento di Matematica ``Tullio Levi Civita''. ORCID: https://orcid.org/0000-0002-2134-4991}
\email{lucchini@math.unipd.it}

\author{Patricia Medina Capilla}
\address{Patricia Medina Capilla, Mathematics Institute , University of Warwick, Coventry, CV4 7AL. ORCID: https://orcid.org/0009-0005-6183-5954}
\email{Patricia.Medina-Capilla@warwick.ac.uk}

\date{May 2025}


\thanks{Project funded by the EuropeanUnion - NextGenerationEU under the National Recovery and Resilience Plan (NRRP), Mission 4 Component 2 Investment 1.1- Call PRIN 2022 No. 104 of February 2, 2022 of Italian Ministry of University and Research; Project 2022PSTWLB (subject area: PE - Physical Sciences and Engineering) ``Group Theory and Applications''. The first author is member of GNSAGA (INDAM). This work was supported by the Additional Funding Programme for Mathematical Sciences, delivered by EPSRC (EP/V521917/1) and the Heilbronn Institute for Mathematical Research.}
\begin{document}
\begin{abstract}
Let  $d(G)$ be the smallest cardinality of a generating set of a finite group $G.$  We give a complete classification of the finite groups with the property that, whenever $	\langle x_1, \dots, x_{d(G)} \rangle = \langle y_1, \dots, y_{d(G)} \rangle = G$,
	for any $1 \leq i \leq d(G)$ there exists  $1 \leq j \leq d(G)$ such that
	$\langle x_1, \dots, x_{i-1}, y_j, x_{i+1}, \dots, x_{d(G)} \rangle = G.$ We also prove that for every finite group $G$ and every maximal subgroup $M$ of $G$, there exists a generating set for $G$ of minimal size  in which at least $d(G)-2$
    elements belong to $M$. We conjecture that the stronger statement holds, that there exists a generating set of size $d(G)$ in which only one element does not belong to $M$, and we prove this conjecture for some suitable choices of $M$.
\end{abstract}

\maketitle

\section{Introduction}
Given a finite group $G$, we denote with $d(G)$ the smallest cardinality of a generating set for $G.$
\begin{definition}
		A group $G$ has the minimal generating set exchange (MGSE) property if 
		\[
		\langle x_1, \dots, x_{d(G)} \rangle = \langle y_1, \dots, y_{d(G)} \rangle = G
		\]
		implies that for any $1 \leq i \leq d(G)$ there exists $1 \leq j \leq d(G)$ such that
		\[
		\langle x_1, \dots, x_{i-1}, y_j, x_{i+1}, \dots, x_{d(G)} \rangle = G.
		\]
	\end{definition}
The purpose of this work is to provide a comprehensive answer to a question proposed by Peter J. Cameron,  Aparna Lakshmanan S and  Midhuna V Ajith \cite[Question 2]{cam}, which, using the previous definition, can be reformulated as follows: {\sl{which finite groups satisfy the MGSE property?}}

Clearly every finite cyclic group satisfies the MGSE property, so we may restrict our attention to non-cyclic finite groups. A partial answer is given in \cite{alars}, where a complete classification of the finite solvable groups $G$ with the MGSE property is obtained. More precisely, if $G$ is nilpotent and non-cyclic then $G$ satisfies the MGSE property if and only if $G$ is a finite $p$-group \cite[Corollary 4.6]{alars}. On the other hand, if $G$ is solvable but not nilpotent, then $G$ satisfies the MGSE property if and only if its quotient over the Frattini subgroup is isomorphic to a semidirect product $N^\delta \rtimes H,$ where $\delta$ is a positive integer, $H$ is cyclic of prime order, $N$ is an irreducible $H$-module and $C_H(V) = 1$ \cite[Theorem 4.13]{alars}.

In this paper, we complete the classification of the finite groups with the MGSE property, proving the following result already conjectured in \cite{alars}.
	
	\begin{theorem} \label{theorem:main}
		If $G$ has the MGSE property, then $G$ is solvable.
	\end{theorem}

    The proof of the previous theorem strongly depends on a somewhat technical result, Lemma \ref{lemma:constructiveargument}, concerning the generation properties of finite monolithic groups with non-abelian socle. The proof of Lemma \ref{lemma:constructiveargument} is obtained by inserting some modifications to the proof of the fundamental theorem in \cite{lm}, which asserts that if the socle $N$ of a monolithic group $G$ is non-abelian, then for every $d\geq 2$ and every subset $\{g_1,...,g_d\}$ such that $G=N\langle g_1,...,g_d\rangle $, there exist $n_1,\dots,n_d$ in $N$ such that $G=\langle g_1n_1,\dots,g_dn_d\rangle.$ The proof contained in \cite{lm} is constructive and highlights a strong degree of freedom in the choice of the elements $n_1,...,n_d.$ The idea of Lemma \ref{lemma:constructiveargument} is that this freedom can be used to prove that there exists a suitable maximal subgroup $M$ of $G$ with the property that for each $1\leq j \leq d,$ one can choose $n_1,...,n_d$ such that all elements $g_in_i$ belong to $M$ except for $g_jn_j.$ The modifications we need to introduce to the proof present in \cite{lm} to obtain Lemma \ref{lemma:constructiveargument} are not many, and much of that proof works without any changes. However, it is a long and complex proof, so we thought it appropriate to prove Lemma \ref{lemma:constructiveargument} in all its details, even at the cost of repeating many of the arguments in \cite{lm}.

 In order to prove Theorem \ref{theorem:main}, we showed that given a generating set $\{ \bar g_1, \dots, \bar g_d \}$ of $G/N$, we can choose preimages $g_i$ of $\bar g_i$ such that 
$\langle g_1, \dots, g_d \rangle = G,$
    while $g_2, \dots, g_d$ all lie in a prescribed maximal subgroup of $G$.
  It is hence a natural question to ask whether, given a group $G$ and a maximal subgroup $M$, we can always find a generating set of $G$ of minimal size such that a certain number of the elements are contained in $M$.

    For a subgroup $H \leq G$, let
    \[
        \mathscr{D}_H(G) = \max_{\langle g_1, \dots, g_{d(G)} \rangle = G} | \{ g_i : g_i \in H \}|,
    \]
    and
    \[
        \mathscr{D}(G) = \min_{M \mleq G} \mathscr{D}_M(G).
    \]
We follow the terminology given in \cite[Definition 1.1.8]{classes}.
A primitive group $G$ is said to be
\begin{enumerate}
\item a primitive group of type 1 if $G$ has an abelian minimal normal subgroup,
\item a primitive group of type 2 if $G$ has a unique non-abelian minimal normal
subgroup,
\item a primitive group of type 3 if $G$ has two distinct non-abelian minimal
normal subgroups.
\end{enumerate}

Let $M$ be a maximal subgroup of a group $G.$ Then $M/M_G$ is
a core-free maximal subgroup of the quotient group $G/M_G.$ Then $M$ is said to
be
\begin{enumerate}
\item a maximal subgroup of type 1 if $G/M_G$ is a primitive group of type 1,
\item a maximal subgroup of type 2 if $G/M_G$ is a primitive group of type 2,
\item a maximal subgroup of type 3 if $G/M_G$ is a primitive group of type 3.
\end{enumerate}

\begin{theorem}Let $M$ be a maximal subgroup of a finite group $G.$
Then $\mathscr{D}_M(G)\geq d(G)-2.$ Moreover $\mathscr{D}_M(G) = d(G)-1$ if one of the following holds:
\begin{enumerate}
    \item $M$ is of type 1;
    \item $M$ is of type 3;
    \item $d(G)>d(G/M_G),$
     \item $G/N$ is cyclic, if $N/M_G=\soc(G/M_G).$
\end{enumerate}
\end{theorem}
\begin{corollary}For every finite group $G,$ $\mathscr{D}(G)\geq d(G)-2$. Moreover $\mathscr{D}(G)=d(G)-1$ if $G$ is solvable.
\end{corollary}

We do not know of any examples of a maximal subgroup $M$ of a finite group $G$ such that $\mathscr{D}_M(G) = d(G)-2.$ This leads to proposing the conjecture that $\mathscr{D}(G)=d(G)-1$ for every finite group $G.$

    \section{Proof of Theorem \ref{theorem:main}}
    Our first observation towards proving Theorem \ref{theorem:main} is the following.
    
	\begin{lemma}\label{lemma:GN}
		If $G$ has the MGSE property, then so does $G/N$ for any $N \lhd G$.
	\end{lemma}
	
	\begin{proof}
		See \cite[Lemma 4.4]{alars}.
	\end{proof}
	
	Since $G$ is non-solvable, there exists a normal subgroup $K$ of $G$ such that $G/K$ is a monolithic group with non-abelian socle. Hence, by Lemma \ref{lemma:GN}, we can assume for the remainder of this proof that $G$ is a monolithic group with socle $N = S^n$ for some non-abelian simple group $S$.

    Thus, we will proceed with some results and definitions specific to monolithic groups with non-abelian socle.

    As $N$ is the unique minimal normal subgroup of $G$, we must have $C_G(N) = 1$. As such,
    \[
    G = N_G(N) = \frac{N_G(N)}{C_G(N)} \leq \Aut(S^n) = \Aut(S) \wr S_n.
    \]
    Hence, $G$ is a subgroup of $\Aut(S) \wr S_n$ containing $S^n$. Additionally, if the map $\pi \colon \Aut(S) \wr S_n \to S_n$ is the natural projection map, then $G \pi$ must be a transitive subgroup of $S_n$.

    \begin{definition} \label{def:Hrpartchange}
        Let $S$ be a non-abelian simple group, and $K$ an almost simple group with socle $S$. We say a subgroup $H$ of $K$ is flexible if it has the following three properties:
        \begin{itemize}
            \item $H \cap S \neq S$;
            \item $H S = K$; and
            \item There exists a prime $r$ such that for every $h \in H$ there exists an $s \in H \cap S$ such that $|h|_r \neq |hs|_r$.
        \end{itemize}
    \end{definition}

    An important fact is that for any $S \leq K \leq \Aut(S)$, there always exists a subgroup $H < K$ which is flexible, by \cite[Lemma 6]{nongen}. The existence of these subgroups is important in our proof of Theorem \ref{theorem:main}, as they allow us the freedom we require in order to provide suitable preimages of a generating set of $G/N$. To show this, we must first obtain a suitable subgroup $\tilde H$ of $G$ which inherits these qualities from $H$.

    We will let $S_i$ denote the subgroup of $S^n$ consisting of all elements $x = (x_1, \dots, x_j)$ such that $x_j = 1$ for all $j \neq i$. Then the preimage $\pi^{-1} \left( \Stab_{S_n}(1) \right)$  has a normal subgroup given by
    \[
        (\Aut(S_2) \times \cdots \times \Aut(S_n) : S_{n-1}) \cap G.
    \]
    Let
    \[
        K_1 := \frac{\pi^{-1} \left( \Stab_{S_n}(1) \right)}{(\Aut(S_2) \times \cdots \times \Aut(S_n) : S_{n-1}) \cap G}.
    \]
    and note we can view $K_1$ as a subgroup of $\Aut(S_1)$.

    \begin{definition} \label{def:Htilde}
        Let $G$ be a monolithic group with non-abelian socle, and define $K_1$ as above. By the Embedding Theorem, see \cite[Theorem 3.3]{bl}, there is an injective homomorphism $\phi \colon G \to K_1 \wr S_n$. Then, for any flexible $H \leq K_1$, we let
        \[
            \tilde{H} := \phi^{-1} \left( \{ (\alpha_1, \dots, \alpha_n) \sigma \in G : \alpha_i \in H \} \right).
        \]
    \end{definition}

	\begin{lemma} \label{lemma:constructiveargument}
		Let $G$ be a monolithic group with non-abelian socle $N$, and take $g_1, \dots, g_{d(G)} \in G$ such that
		\[
		\langle g_1, \dots, g_{d(G)} \rangle N = G/N.
		\]
		
		If one of the following holds:
		\begin{itemize}
			\item $g_1 \pi$ has a fixed point,
			\item $g_2 \pi$ has a fixed point, or
			\item $(g_1^{-1} g_2^{i}) \pi$ and $(g_2^{-1} g_1^{i}) \pi$ are fixed point free for all $i$,
		\end{itemize}
		then there exist $v_1, v_2 \in N$ such that
		\[
		\langle v_1 g_1, v_2 g_2, g_3 \dots, g_{d(G)} \rangle = G.
		\]
		Moreover, for any flexible subgroup $H \leq K_1$, we can choose $v_1$ such that $g_1 v_1 \in \tilde{H}$, with $\tilde{H}$ as defined in Definition \ref{def:Htilde}.
	\end{lemma}
	
	If we assume the statement of this lemma, then we may prove the main theorem of this paper.
	
	\begin{proof}[Proof of Theorem \ref{theorem:main}]
		Let $G$ be a non-solvable group, and assume $G$ has the MGSE property. By Lemma \ref{lemma:GN}, we can assume that $G$ is a monolithic group with socle $N = S^n$, for some non-abelian simple group $S$.

        Let $J \leq \Aut(S)$ be a subgroup as given in \cite[Lemma 6]{nongen}, which is flexible in $\Aut(S)$. Thus, $H := J \cap K_1$ is flexible in $K_1$. Hence, we may define $\tilde H$ as in Definition \ref{def:Htilde}.
		
		Let $d := d(G)$, and choose $g_1, \dots, g_d \in \tilde H$ such that
		\[
		\langle g_1, \dots, g_d \rangle N = G/N.
		\]
		This is possible since $HS = K_1$, which implies that $\tilde H N = G$.
		
		By \cite{lm}, $d(G) = \max(2, d(G/N))$. If $d(G/N) = 1$, then we choose $g_1$ such that $g_1 N$ is a generator of $G/N$, and $g_2 = 1$. Hence, we have $\langle g_2, \dots, g_d \rangle N < G$. On the other hand, if $d(G/N) > 1$, then $d = d(G/N)$, and so we also have $\langle g_2, \dots, g_d \rangle N < G$.
		
		In order to apply Lemma \ref{lemma:constructiveargument}, we must satisfy one of the three conditions in the statement. Suppose that $(g_1^{-1} g_2^i) \pi$ has a fixed point for some $i$. Then, we can instead consider the generating set
		\[
		\langle g_1^{-1} g_2^i, g_2, \dots, g_d \rangle N = G/N,
		\]
		and proceed with applying Lemma \ref{lemma:constructiveargument} now. Note, $g_1^{-1} g_2^i \in \tilde{H}$ still. We can similarly replace $g_2$ by $g_2^{-1} g_1^i$ instead if this element has a fixed point for some $i$.
        
        Additionally, we still have the property that $\langle g_2, \dots, g_d \rangle N < G$. If $d \geq 2$ then this is clear by $d(G) = \max(2, d(G/N))$. Otherwise, if $d(G/N) = 1$, then $g_1 \pi$ must be an $n$-cycle. Hence, $g_1^{-1} g_2^i$, is always fixed-point free, and $g_2^{-1} g_1^i$ is fixed point free, unless it is equal to 1.

        Hence, we can assume we have chosen $g_1, \dots, g_d \in \tilde{H}$ such that
        \[
            \langle g_1, \dots, g_d \rangle N = G,
        \]
        with $g_1, g_2$ satisfying one of the conditions of Lemma \ref{lemma:constructiveargument}, and $\langle g_2, \dots, g_d \rangle N < G$.
		
		Thus, we can find $v_1, v_2 \in N$ and $w_1, w_2 \in N$ such that
		\[
		\langle v_1 g_1, v_2 g_2, g_3, \dots, g_d \rangle = \langle w_1 g_1, w_2 g_2, g_3, \dots, g_d \rangle = G,
		\]
		and $v_2 g_2, w_1 g_1 \in \tilde{H}$. Since $G$ has the MGSE property, there exists a $j$ such that 
		\[
		\langle w_j g_j, v_2 g_2, g_3, \dots, g_d \rangle = G,
		\]
		where $w_j := 1$ for $j \geq 3$. If $j > 1$, then this implies that $\langle g_2, \dots, g_d \rangle N = G/N$, contradicting our choice of $g_i$ above. If $j=1$ instead, then
		\[
		\langle w_1 g_1, v_2 g_2, g_3, \dots, g_d \rangle \leq \tilde{H} < G,
		\]
		again giving a contradiction. Hence, $G$ cannot have the MGSE property.
	\end{proof}

    For the proof of Lemma \ref{lemma:constructiveargument}, the following definition is useful.

    \begin{definition} \label{def:quasiordering}
        Let $r$ be a prime, and $\sigma_1, \sigma_2 \in S_n$. We define a semi-ordering on $S_n$ given by: $\sigma_1 \leq \sigma_2$ if and only if $|\sigma_1|_r < |\sigma_2|_r$ or $|\sigma_1|_r = |\sigma_2|_r$ and $|\sigma_1| \leq |\sigma_2|$.
    \end{definition}

    We will now prove Lemma \ref{lemma:constructiveargument}. As mentioned in the introduction, a version of this lemma without the extra condition that $g_1 v_1 \in \tilde{H}$ is proved in \cite{lm}. The proof of Lemma \ref{lemma:constructiveargument} follows in nearly the exact same way, due to the argument in \cite{lm} providing enough freedom in the choice of $v_1, v_2$ to ensure $g_1 v_1 \in \tilde H$.

	\begin{proof}[Proof of Lemma \ref{lemma:constructiveargument}]
        Let $G$ be a monolithic group with socle $N = S^n$, for some non-abelian simple group $S$.
    
		Define $S \leq K_1 \leq \Aut(S) $ for $G$ as in the argument before Definition \ref{def:Htilde}, and $\phi$ the injective map $\phi \colon G \to K_1 \wr S_n$ given by \cite[Theorem 3.3]{bl}. Let $H \leq K_1$ be a flexible subgroup for some prime $r$, and
        \[
            \tilde H := \phi^{-1} \left( \{(\alpha_1, \dots, \alpha_n) \sigma \in G : \alpha_i \in H \} \right) < G
        \]
        as in Definition \ref{def:Htilde}. We will use the quasi-ordering defined on $S_n$ in Definition \ref{def:quasiordering}, with $r$ as the prime in the definition.

        We note that in order to prove Lemma \ref{lemma:constructiveargument}, it is sufficient to prove that there exist $v_1, v_2 \in N^\phi = S^n$ such that
        \[
            \langle v_1 g_1, v_2 g_2, g_3, \dots, g_{d(G)} \rangle = G^\phi
        \]
        with $v_1 g_1 \in \tilde{H}^\phi$. Hence, we will assume that $N \leq G \leq K_1 \wr S_n$, and that
        \[
            \tilde H = \{(\alpha_1, \dots, \alpha_n) \sigma \in G : \alpha_i \in H \}.
        \]
		
		As previously, we will let $S_i$ denote the subgroup of $S^n$ consisting of all elements $x = (x_1, \dots, x_j)$ such that $x_j = 1$ for all $j \neq i$. Let $\pi \colon K_1 \wr S_n \to S_n$ denote the natural projection map, and $\pi_i \colon K_1^n \to K_1$ the projection maps from the base group to the $i$th coordinate.
		
		Let $g_1, \dots, g_d$ be as in the statement of the lemma, with $d = d(G)$. Since $\tilde{H}$ projects onto $G/N$, we can assume that $g_i \in \tilde{H}$ for all $i$. Note, we can let $g_1 = \alpha \rho$ and $g_2 = \beta \sigma$, where $\alpha, \beta \in K_1^n$, and $\rho, \sigma \in S_n$. We will choose elements $x, y \in S^n$, and will consider the subgroup $\langle x g_1, y g_2, g_3, \dots, g_d \rangle \leq G$. Our goal is to ensure this subgroup is the entirety of $G$ for some $x \in \tilde{H}$. Let $x = (x_1, \dots, x_n)$ and $y = (y_1, \dots, y_n)$.
		
		We can write $\rho = \rho_1 \cdots \rho_{s(\rho)}$ as a product of disjoint cycles, including possible cycles of length 1, such that
		\[
		\rho_1 \leq \cdots \leq \rho_{s(\rho)}.
		\]
		Note, $\rho$ has a fixed point point if and only if $|\rho_1| = 1$. 
		
		We similarly write $\sigma = \sigma_1 \cdots \sigma_q \cdots \sigma_{s(\sigma)}$ as a product of disjoint cycles such that:
		\begin{itemize}
			\item $\supp(\rho_1) \cap \supp(\sigma_i) \neq 1$ if and only if $i \leq q$;
			\item $\sigma_1 \leq \cdots \leq \sigma_q$.
		\end{itemize}
		
		For each $i$, let
		\[
		\rho_i := (m_{i,1}, \dots, m_{i, |\rho_i|}),
		\]
		and set $m_{1,1} =: m$. Additionally, define $\bar{\alpha}_i := x_i \alpha_i$, and
		\[
		a_i = \bar{\alpha}_{m_{i,1}} \cdots \bar{\alpha}_{m_{i, |\rho_i|}},
		\]
		which is the $m_{i,1}$th entry of $(x g_1)^{|\rho_i|}$. Let $a := a_1$.
		
		There exists a $\sigma_j$ with $j \leq q$ such that $m$ is in the support of $\sigma_j$. As for $\alpha$, we fix $\sigma_i = (n_{i,1}, \dots, n_{i, l_i})$, and
		\[
		\bar \beta_i := y_i \beta_i, \qquad b_i := \bar \beta_{n_{i,1}} \cdots \bar \beta_{n_{i,l_i}},
		\]
		for $2 \leq i \leq q$. Let $b := b_1$.
		
		We say $x, y \in S^n$ form a good pair if the following two conditions hold:
		\begin{enumerate}
			\item $\langle a, b \rangle \geq S$; and
			\item For any $2 \leq i \leq s(\rho)$, $a_i^{|\rho_1 \cdots \rho_i|/{|\rho_i|}}$ is not conjugate to $a_1^{|\rho_1 \cdots \rho_i|/{|\rho_1|}}$ in $\Aut(S)$.
		\end{enumerate}
		
		We will prove that for each good pair $x, y$, there exists a partition $\Phi$ of $\{1, \dots, n\}$ fixed by $g_1, \dots, g_d$ such that 
		\[
		\langle x g_1, y g_2, g_3 \dots, g_d \rangle \cap S^n = \prod_{B \in \Phi} D_B
		\]
		where for each $B \in \Phi$, $D_B$ is the full diagonal subgroup of $\prod_{i \in B} S_i$. Moreover, if $B_1$ is the block of $\Phi$ containing $m$, then $B_1 \subseteq \supp(\rho_1) $. Hence, if $\rho$ has a fixed point, then we're done, given we can find a good pair $x, y$ with $x \in \tilde H$.
		
		Let
		\[
		K := \langle x g_1, y g_2, g_3 \dots, g_d \rangle.
		\]
		
		The key thing to note is that $(x g_1)^{|\rho_1|}, (y g_2)^{|\sigma_j|}$ normalise $S_m$. Additionally, their $m$th coordinates are $a, b$ respectively, by definition. Hence, $(K \cap K_1^n)\pi_m$ must be normal in $\langle a, b \rangle \geq S$. Thus, it must be either trivial, or contain $S$.
		
		However, this projection cannot be trivial. Indeed, we can consider $(xg_1)^{|\rho|} \in K \cap K_1^n$. Firstly, if $\rho$ is an $n$-cycle, then $(xg_1)^{|\rho|} \pi_{m} = a_1$. Since $\langle a, b \rangle \geq S$, we must have $a_1 \neq 1$. Thus, $(K \cap K_1^n) \pi_m \neq 1$.
        
        If instead $\rho$ is not an $n$-cycle, then $s(\rho) > 1$, and
		\[
		(xg_1)^{|\rho|} \pi_{m} = a_1^{|\rho|/|\rho_1|}, \qquad (xg_1)^{|\rho|} \pi_{m_{s(\rho)}} = a_{s(\rho)} ^{|\rho|/|\rho_{s(\rho)}|}.
		\]
        Since $a_1^{|\rho|/|\rho_1|}, a_{s(\rho)} ^{|\rho|/|\rho_{s(\rho)}|}$ are not conjugate, we can immediately exclude the possibility that $(xg_1)^{|\rho|} = 1$. Therefore, $(xg_1)^{|\rho|} \pi_k \neq 1$ for some $k$, and so the same is true for every $k \in \{1, \dots, n \}$, since $G \pi$ is a transitive subgroup of $S_n$.
		
		Hence, the projection $(K \cap K_1^n) \pi_m$ must contain $S$. Thus, by the transitivity of $G \pi$, for each $i \in \{1, \dots, n\}$ we have $(K \cap K_1^n) \pi_i \geq S$. Since $\Aut(S)/S$ is solvable, this implies that $(K \cap S^n ) \pi_i = S$ for every $i$. Therefore,
        \[
            K \cap S^n = \prod_{B \in \Phi} D_B,
        \]
        for some partition $\Phi$ of $\{1, \dots, n\}$.
		
		Finally, if $B$ is the block containing $m$, then we will show that $B \subseteq \supp(\rho_1)$. Suppose for a contradiction that there exists some $h \in B$ which is not in $\supp(\rho_1)$. Let $h = m_{j,t} \in \supp(\rho_j)$. We may assume that
		\[
		D_B = \{ (x, x^{\phi_h}, \dots) : x \in S \} \leq S_m \times S_h \times \cdots.
		\]
		for some $\phi_h \in \Aut(S)$. Since $(xg_1)^{|\rho_1 \cdots \rho_h|} \pi$ stabilises $m$ and $h$, this implies $(xg_1)^{|\rho_1 \cdots \rho_h|}$ normalises $D_B$. However,
		\[
		(x, x^{\phi_h}, \dots)^{(xg_1)^{|\rho_1 \cdots \rho_h|}} = (x^{\lambda_m}, x^{\phi_h \lambda_h}, \dots),
		\]
		where
		\[
		\lambda_m = a_1 ^{|\rho_1 \cdots \rho_h|/|\rho_1|},
		\]
		and
		\begin{align*}
			\lambda_h & = (\bar \alpha_{m_{j, t}} \cdots \bar \alpha_{m_{j, k_j}} \bar \alpha_{m_{j,1}} \cdots \bar \alpha_{m_{j, t-1}})^{|\rho_1 \cdots \rho_h|/|\rho_h|}\\
			& = (\bar \alpha_{m_{j,1}} \cdots \bar \alpha_{m_{j, t-1}})^{-1} a_j ^{|\rho_1 \cdots \rho_h|/|\rho_h|} (\bar \alpha_{m_{j,1}} \cdots \bar \alpha_{m_{j, t-1}}).
		\end{align*}
		Hence, since $(xg_1)^{|\rho_1 \cdots \rho_h|}$ normalises $D_B$, we must have that
		\[
		\phi_h^{-1} a_1 ^{|\rho_1 \cdots \rho_h|/|\rho_1|} \phi_h = (\bar \alpha_{m_{j,1}} \cdots \bar \alpha_{m_{j, t-1}})^{-1} a_j ^{|\rho_1 \cdots \rho_h|/|\rho_h|} (\bar \alpha_{m_{j,1}} \cdots \bar \alpha_{m_{j, t-1}}),
		\]
		contradicting (2).
		
		Thus, we have shown the desired result. We will now show we can always find a good pair $x, y \in S^n$ such that $x \in \tilde{H}$. Consider $x, y \in S^n$, with $x \in \tilde{H}$. We will modify these to obtain a good pair.
		
		By \cite[Lemma 7]{nongen}, we can always find $g \in S \cap H, h \in S$ such that $\langle g a, h b \rangle \geq S$. Hence, substituting $x_m$ for $g x_m$ and $y_m$ for $h y_m$ is sufficient to satisfy (1), and we will have $x \in \tilde{H}$.
		
		Next, by \cite[Lemma 6]{nongen}, since $|\rho_1 \cdots \rho_i|/|\rho_i|$ is coprime to $r$, there exists an $s_i \in S \cap H$ such that $(s_i a_i)^{|\rho_1 \cdots \rho_i|/|\rho_i|}$ is not conjugate to $a_i ^{|\rho_1 \cdots \rho_i|/|\rho_i|}$. Thus, $a_1^{|\rho_1 \cdots \rho_i|/|\rho_1|}$ will not be conjugate to one of these two elements, and hence we can choose an $s_i \in S \cap H$ such that $(s_i a_i)^{|\rho_1 \cdots \rho_i|/|\rho_i|}$ is not conjugate to $a_1^{|\rho_1 \cdots \rho_i|/|\rho_1|}$. Thus, if we substitute $x_{m_{i,1}}$ with $s_i x_{m_{i-1}}$, we will obtain that
        \[
            a_i^{|\rho_1 \cdots \rho_i|/|\rho_i|} \text{ is not conjugate to } a_1^{|\rho_1 \cdots \rho_i|/|\rho_1|}.
        \]

        We repeat this for each $i > 1$, noting that at each step we still have $x \in \tilde H$.

        Hence, we can always find a good pair $x,y$ with $x \in \tilde H$, and so we may assume that
        \[
            K \cap S^n = \langle x g_1, y g_2, g_3 , \dots, g_d \rangle \cap S^n = \prod_{B \in \Phi} D_B
        \]
        for some partition $\Phi$ such that the block containing $m$ is contained in $\supp(\rho_i)$.
		
		Hence, if $\rho$ has a fixed point, then we're done, since $B \subseteq \supp(\rho_1) = \{m\}$, which implies $K \cap S^n = S^n$ and thus $K = G$. If instead $\sigma$ had a fixed point, then we could've repeated the same process with $(g_2, g_1, g_3, \dots, g_d)$ instead. However, if we repeated the argument in the exact same manner, we would obtain $x g_2 \in \tilde{H}$, whereas we would like $y g_1 \in \tilde{H}$. This is easily remedied by noting that we could choose either $g$ or $h$ to be contained in $H \cap S$ in our application of \cite[Lemma 7]{nongen}, and hence the opposite choice from above would mean $y g_1 \in \tilde{H}$ as required.

        Thus, we may now assume that $(g_1^{-1} g_2^i) \pi$ and $(g_2^{-1} g_1^i) \pi$ are both fixed point free. In the remainder of this proof, we will only modify $y$. We will first show that we can change $y$ in such a way that $B \subseteq \supp(\rho_1) \cap \supp(\sigma_1)$. Then, we will show that for any such block, $D_B$ and $D_{B \rho}$ are entirely determined. Finally, we will show we can modify $y$ so that $x g_1$ and $y g_2$ normalise $D_{B \rho}$ if and only if $|B| = 1$. This will hence prove the result.

        Firstly, as with $\alpha$, we can modify $y \in N$ such that
        \[
            b_i^{|\sigma_1 \cdots \sigma_i|/|\sigma_i|} \text{ is not conjugate to } b_1^{|\sigma_1 \cdots \sigma_i|/|\sigma_1|} \qquad \forall 2 \leq i \leq q,
        \]
        without changing the value of $b_1$.
		
		The proof that $B \subseteq \supp(\sigma_1)$ now proceeds similarly to the proof that $B \subseteq \supp (\rho_1)$, since
		\[
		B \subseteq \supp(\rho_1) \subseteq \supp(\sigma_1) \cup \cdots \supp(\sigma_q).
		\]
		Indeed, if $h \in B \setminus \supp(\sigma_1)$, then $h \in \supp(\sigma_i)$ for $2 \leq i \leq q$. Thus, as before, this would imply that $b_1^{|\sigma_1 \cdots \sigma_i|/|\sigma_1|}$ is conjugate to $b_i^{|\sigma_1 \cdots \sigma_i|/|\sigma_i|}$, leading to a contradiction.
		
		A consequence of this is that $B \rho \cap \supp(\sigma_1) = \emptyset$. Suppose $h \in B \rho \cap \supp(\sigma_1) $, so $h = j \rho$ for for $j \in B \subseteq \supp(\sigma_1)$. Hence, there exists an $i$ such that $j \rho = h = j \sigma_1^i = j \sigma^i$. Thus, $j$ is a fixed point of $\rho^{-1} \sigma^i$, contradicting our assumption.
		
		In particular, this implies that $B \cap B \rho = \emptyset$.
		
		By the results above, $|B| = c$ is a divisor of $k = |\rho_1|$. Hence,
		\[
			B = \{ m_1, m_{k/c + 1}, \dots, m_{k(c-1)/c + 1} \}
		\]
		which is the orbit of $m_1$ under $\rho_1^{k/c}$. Additionally, we will let
		\[
			D_B = \{ (x, x^{\phi_2}, \dots, x^{\phi_c}) : x \in S \}
		\]
		for some $\phi_i \in \Aut(S)$.

        We will now show that for any $B \subseteq \supp(\rho_1) \cap \supp(\sigma_1)$, the subgroup $D_B$ is uniquely defined by certain $x_i, y_i, \alpha_i, \beta_i$.
		
		Let $t_i := k(i-1)/c + 1$. Since $m_{t_i} \in B \subseteq \supp(\sigma_1)$, there exist some $1 \leq u_i \leq |\sigma_1| = l$ such that $m_{t_i} = n_{u_i}$.
		
		The elements $g_1^{|\rho_1|}$ and $g_2^{|\sigma_1|}$ normalise $D_B$, and $S_{m_{t_i}}$ for all $i$. Hence, there exist $\lambda_i, \mu_i$ such that
		\begin{align*}
			(x, x^{\phi_2}, \dots, x^{\phi_c})^{g_1^{|\rho_1|}} & = (x^{\lambda_1}, x^{\phi_2 \lambda_2}, \dots, x^{\phi_c \lambda_c})\\
			(x, x^{\phi_2}, \dots, x^{\phi_c})^{g_2^{|\sigma_1|}} & = (x^{\mu_1}, x^{\phi_2 \mu_2}, \dots, x^{\phi_c \mu_c}),
		\end{align*}
		which are both in $D_B$. As such,
		\begin{align*}
			\lambda_i & = \phi_i^{-1} \lambda_1 \phi_i = a_1^{\phi_i}\\
			\mu_i & = \phi_i^{-1} \mu_1 \phi_i = b_1^{\phi_i}.
		\end{align*}
		Thus, if $\lambda_i$, $\mu_i$ are fixed, then $\phi_i$ are also fixed. Indeed, as $\langle a_1, b_1 \rangle \geq S$, there can exist at most 1 automorphism of $S$ such that the above holds.
		
		However, we have
		\begin{align*}
			\lambda_i & = \prod_{t_i \leq j \leq k} \bar \alpha_{m_j} \prod_{1 \leq j \leq t_i-1} \bar \alpha_{m_j}\\
			\mu_i & = \prod_{u_i \leq j \leq k} \bar \beta_{n_j} \prod_{1 \leq j \leq u_i-1} \bar \beta_{n_j}.
		\end{align*}
		As such, $\lambda_i, \mu_i$ are determined by the value of $x_i, \alpha_i, y_j, \beta_j$ for $i \in \supp(\rho_1)$ and $j \in \supp(\sigma_1)$. For the remainder of this proof, we will not change the value of these elements. Hence, we can assume that for each choice of $B \subseteq \supp(\rho_1) \cap \supp(\sigma_1)$, there exists a unique $D_B$ normalised by $\langle x g_1, y g_2 \rangle$.

        As an immediate consequence, since
		\[
			D_{B \rho} = D_B^{g_1},
		\]
		the diagonal group $D_{B \rho}$ is also uniquely determined for the remainder of this proof. We have $B \rho = \{m_2, \dots, m_{t_c + 1}\}$, and will let
		\[
			D_{B \rho} = \{ (x, x^{\phi_2^*}, \dots, x^{\phi_c^*}) : x \in S\}.
		\]
		
		We will finally modify $y_i$ for $i \not\in \supp(\sigma_1)$ in order to prove that $\langle x g_1, y g_2 \rangle$ stabilises $\prod_{B \in \Phi} D_B$ if and only if it is equal to $S^n$.

        We will consider
		\begin{align*}
			b_{h,s} = \bar \beta_{n_{h,s}} \cdots \bar \beta_{n_{h, l_h}} \bar \beta_{n_{h,1}} \cdots \bar \beta_{n_{h,s-1}}.
		\end{align*}
		Note, $b_{h,1} = b_h$, and that $b_{h,s}$ is the $m_{h,s}$th coordinate of $g_2^{|\rho_h|}$.

        Let $i$ be such that $m_2 \in \supp(\sigma_i)$. Note $i \neq 1$, since $B \rho \cap \supp(\sigma_1) = \emptyset$.
        
		We will first consider the $c \mid k$ with $m_{t_c + 1} \in \supp(\sigma_i)$. We know that $g_2^{|\sigma_i|}$ normalises $D_{B \rho}$, and fixes the coordinates $m_2$ and $m_{t_c + 1}$. Since $D_{B \rho}$ is uniquely determined, we can prove that the action of $g_2^{|\sigma_i|}$ on $S_{m_2}$ must determine its action on $S_{m_{t_c} + 1}$. Thus, if $m_2 = n_{i,p}$ and $m_{t_c + 1} = n_{i,q}$, then $b_{i,p}$ determines $b_{i,q}$. However, we can alter some $y_{n_{i,j}}$ for certain $j$ in order to replace $b_{i,q}$ with a conjugate of itself whilst not altering any other $b_{i,j}$, which gives a contradiction.

        Notice that we might exchange $b_i$ for a conjugate of itself in these modifications, but then we will still have $b_i^{|\rho_1 \cdots \rho_i|/|\rho_i|}$ not conjugate to $b_1^{|\rho_1 \cdots \rho_i|/|\rho_1|}$. 

        We do this for each $c \mid k$ such that $m_{t_c + 1} \in \supp(\sigma_i)$. Each $c$ gives rise to a different $m_{t_c+1}$, and thus to a different $q$. Hence, each step does not affect the previous, and we have shown that we can choose $x,y$ such that $\langle x g_1, y g_2 \rangle$ does not normalise $D_{B \rho}$ for any such $c$.

        Finally, we treat $c \mid k$ with $m_{t_c + 1} \in \supp(\sigma_h)$ for some $h \neq i$. Note also that $h \neq 1$, since $m_{t_c + 1} \in B \rho$ has trivial intersection with $\supp(\sigma_1)$. In this case, $g_2^{|\sigma_h|}$ fixes the $m_{t_c + 1}$ coordinate, and thus normalises $D_{B \rho}$. However, it does not fix the $m_2$ coordinate, as in the previous case. Despite this, we still have
        \[
            (x, x^{\phi_2^*}, \dots, x^{\phi_c^*})^{g_2^{|\sigma_h|}} = (x^\gamma, \dots, x^{\phi_c^* b_{h,q}})
        \]
        where $\gamma$ is determined by $\phi_j^*$ and $\bar \beta_s$ for $s \in \supp(\sigma_i)$. Thus, once again $b_{h,q}$ is determined, but we can alter $y$ in order to replace $b_{h,q}$ with a conjugate of itself, while fixing all other $b_{h,j}$, leading to a contradiction.

        We can repeat this, as before, for all such $c \mid k$, and none of these alterations to $y$ will affect previous steps.

        Hence, we must have $m_2 = m_{t_c + 1}$, or equivalently, $|B| = 1$, which proves the result.
	\end{proof}

    As a consequence of the above proof, we obtain the following two corollaries. Given $\bar g_1, \dots, \bar g_d$ which generate $G/N$, these corollaries illustrate the freedom with which one can choose preimages $g_1, \dots, g_d$ generating $G$.

    \begin{corollary}
        Let $G$ be a monolithic group with socle $N$, and suppose $g_1, g_2$ are as in the statement of Lemma \ref{lemma:constructiveargument}. Then there exists some $v_1, v_2 \in N$ such that 
        \[
            \langle v_1 g_1, v_2 g_2, h_3, \dots, h_d \rangle = G
        \]
        for all $h_i$ such that
        \[
            \langle g_1, g_2, h_3, \dots, h_d \rangle N = G.
        \]
    \end{corollary}

    \begin{corollary} \label{cor:alterg2}
        Let $G$ be a monolithic group with socle $N$, and suppose $g_1, \dots, g_d$ are as in Lemma \ref{lemma:constructiveargument}. Let 
        \[
            g_1 = (\alpha_1, \dots, \alpha_n) \rho, \qquad g_2 = (\beta_1, \dots, \beta_n) \sigma.
        \]
        As in the proof of Lemma \ref{lemma:constructiveargument}, define $a_i$ and $b_1$.
        
        Suppose
        \[
            |a_1|_r \neq |a_i|_r
        \]
        for all $i > 1$, and that there exists $s \in S$ such that $\langle a_1, sb_1 \rangle \geq S$. Then there exists $v \in N$ such that
        \[
            \langle g_1, v g_2, \dots, g_d \rangle = G.
        \]
    \end{corollary}

    \section{Generators and maximal subgroups}

    In order to prove Theorem \ref{theorem:main}, we proved we can always find a generating set $g_1, \dots, g_{d(G)}$ of $G$ with $g_2, \dots, g_{d(G)}$ contained in a particular proper subgroup of $G$. Thus, we pose the following question: given a maximal subgroup $M$ of a finite group $G$, is it always possible to provide a generating set of minimal size with a given number of elements in $M$?

    \begin{definition}
    For a subgroup $H \leq G$, let
    \[
        \mathscr{D}_H(G) = \max_{\langle g_1, \dots, g_{d(G)} \rangle = G} | \{ g_i : g_i \in H \}|,
    \]
    and
    \[
        \mathscr{D}(G) = \min_{M \mleq G} \mathscr{D}_M(G).
    \]
    \end{definition}

    Hence, we would like to bound $\mathscr{D}(G)$ for any finite group $G$. To do so, we first recall the terminology given in \cite[Definition 1.1.8]{classes}. A primitive group $G$ is said to be
    \begin{enumerate}
    \item a primitive group of type 1 if $G$ has an abelian minimal normal subgroup,
    \item a primitive group of type 2 if $G$ has a unique non-abelian minimal normal
    subgroup,
    \item a primitive group of type 3 if $G$ has two distinct non-abelian minimal
    normal subgroups.
    \end{enumerate}
    
    Let $M$ be a maximal subgroup of a group $G.$ Then $M/M_G$ is
    a core-free maximal subgroup of the quotient group $G/M_G.$ Then $M$ is said to
    be
    \begin{enumerate}
    \item a maximal subgroup of type 1 if $G/M_G$ is a primitive group of type 1,
    \item a maximal subgroup of type 2 if $G/M_G$ is a primitive group of type 2,
    \item a maximal subgroup of type 3 if $G/M_G$ is a primitive group of type 3.
    \end{enumerate}

    We treat each type of maximal subgroup $M$ individually in the following propositions.

\begin{prop}
If $M$ is a maximal subgroup of a finite group $G$ of type 3, then $\mathcal D_M(G)=d(G)-1.$
\end{prop}
\begin{proof}
It is not restrictive to assume $M_G=1$. Then there exists a monolithic primitive group $L$ with $N=\soc L$ non-abelian, and we may identify $G$ with the subgroup of $L^2$ consisting of the pairs $(l_1,l_2)$ with $l_1 \equiv l_2 \mod N$ and $M$ with its diagonal subgroup $\{(l,l)\mid l \in L\}.$ 

Assume  $N=S^n,$ $S$ simple non-abelian, $d=d(L)$. Choose $(l_1,\dots,l_d)$ such that $G=\langle l_1,\dots, l_d\rangle N$
and let $$\Lambda=\{(n_1,\dots,n_d)\in N^d\mid \langle l_1n_1,\dots,l_dn_d\rangle=L.$$ By \cite[Theorem 1.1]{delu},
$|\Lambda|\geq \frac{53}{90}|N|^d.$ This implies that we may choose $l_1,\dots,l_d$ with the property that $\langle l_1,\dots,l_d\rangle =L$, and the set
    $$\Omega=\{n \mid \langle l_1n,l_2,\dots,l_d\rangle=L\}$$
    has size at least $\frac{53}{90}|N|$.
    
    For a given $n\in \Omega$ consider the following subset of $G:$
    $$x_1=(l_1,l_1n), x_2=(l_2,l_2),\dots, x_d=(l_d,l_d).$$
    Let $X=C_{\Aut(N)}(L/N)$.  If $\langle x_1,\dots,x_d\rangle \neq L_2,$ then there exists $\gamma \in \Aut(N)$
    such that  $x_1,\dots,x_d$ normalize the diagonal subgroup $\Delta_\gamma=\{(z,z^\gamma)\mid z \in N\}$ of $G.$
 This implies $l_1n=l_1^\gamma,$ $l_2=l_2^\gamma,\dots,l_d=l_d^\gamma.$
    In particular $\gamma \in Y:=C_X(l_2,\dots,l_d).$
     Assume, for a contradiction, that there exists such an automorphism $\gamma$
    for every $n\in \Omega.$ This implies
    $l_1\Omega \subseteq l_1^Y$, and therefore
    $$|C_X(l_2)|\geq |Y|\geq \frac{53}{90}|N|.$$
   We prove that the previous inequality cannot occur.  If follows from the final part of the proof of \cite[Lemma 1]{pr} that
    $X^\pi$ is a semiregular subgroup of $S_n$ and $|\ker \pi \cap X|\leq |S|^n|\Out(S)|.$ Let $\sigma=l_2^\pi$ and $b$ the number of fixed points of $\sigma.$ Write $n=a+b.$
    Since $X^\pi$ is semiregular, $(C_X(l_2))^\pi$ is isomorphic to a semiregular subgroup of $S_a.$ Let $Z=C_X(l_2)\cap N$. By the same argument as in \cite[Lemma 1]{pr}, we can show that,
    \[
        Z \leq \prod_{\Omega \in \{1,\dots, n\}/\langle \sigma \rangle} \diag \left( S^{|\Omega|} \right).
    \]
    Hence,
    \[
        |C_X(l_2) \cap \Aut(S)^n| \leq |S|^{b + \frac{a}{2}} |\Out(S)|.
    \]
    If $a \neq 0$, this implies that
    \begin{align*}
        |C_X(l_2)| & \leq |S|^{b + \frac{a}{2}} |\Out(S)| a\\
        & \leq \frac{a|S|^{b + \frac{a}{2} + 1}}{30},
    \end{align*}
    where the second inequality is by \cite[Lemma 2.2]{q}. We can show this is always strictly less than $|S|^n/2$, contradicting the inequality above.

    So assume $a = 0$, and $l_2 = (z_1, \dots, z_n)$. If $c$ is the number of $i$ with $z_i \neq 1$, then
    \[
        |Z| \leq \frac{|S|^n}{k^c}
    \]
    where $k$ is the smallest index of a proper subgroup of $S.$
    By \cite[Lemma 2.7]{ag}, $k\geq 3\,|\!\Out(S)|/2$
    and therefore, since $c \neq 0,$
    \[
        |C_X(l_2)|\leq \frac{2c |N|}{3k^{c-1}}.
    \]
    However, if $c > 1$, then this contradicts the inequality above. So we may assume $l_2=(z,1,\dots,1).$ If $S\neq A_6$
    we can use the stronger inequality $k\geq 2\,|\!\Out(S)|$ and deduce a contradiction. If $S=A_6$ and $n > 1$, then there exists $s\in S$ with $(|z|,|s|)=1$ and we may replace $l_2$ with $(z,s,\dots,s).$ Otherwise, if $S = A_6$ and $n=1$ then we can check computationally that the result holds.
\end{proof}

\begin{prop}
If $M$ is a maximal subgroup of a finite group $G$ of type 1, then $\mathcal D_M(G)=d(G)-1.$
\end{prop}

\begin{proof}We may assume $M_G=1.$ Let $N=\soc G$ and $d=d(G).$
Choose $m_1,\dots,m_d$ such that $M=\langle m_1,\dots,m_d\rangle.$ By \cite[Lemma 1]{lmsolvable}, there exist $1\leq j\leq d$ and $n\in N$ such that $G=\langle m_1,\dots,m_{j-1},m_jn,m_{j+1},\dots,m_d\rangle.$
\end{proof}

\begin{prop}\label{prop:type2maximalbound}
    If $M$ is a maximal subgroup of a finite group $G$ of type 2, then $\mathscr{D}_M(G) \geq d(G) - 2$.
\end{prop}

\begin{proof}
    We may assume $M_G = 1$. Then, by definition, $G$ is a monolithic group, and $M$ is a maximal subgroup which projects onto $G / \soc(G)$. Hence, we can choose $g_1, \dots, g_{d(G)}$ in $M$ such that 
    \[
        \langle g_1, \dots, g_{d(G)} \rangle \soc(G) = G.
    \]
    The result is hence clear by Lemma \ref{lemma:constructiveargument}, using a similar argument as in the proof of Theorem \ref{theorem:main}.
\end{proof}

\begin{prop}Assume that $M$ is a maximal subgroup of a finite group $G$ of type 2 and let $N/M_G=\soc(G/M_G).$ If $G/N$ is cyclic, then $\mathcal D_M(G)=d(G)-1.$
\end{prop}

\begin{proof}We may assume $M_G=1.$ Choose $g\in G$ such that $\langle g, N\rangle=G.$ Since $MN=G$ and, as a consequence of the O'Nan-Scott Theorem, $M\cap N\neq 1,$ we may assume $1\neq g\in M.$ By \cite[Theorem 1]{bgh}, there exists $x\in G$ such that $G=\langle g,x\rangle.$
\end{proof}

We additionally prove the following proposition, which allows us to show that $\mathscr{D}_M(G) = d(G) - 1$ in the specific situation where $d(G) > d(G/M_G)$.

\begin{prop}
    Let $G$ be a monolithic group with non-abelian socle, and let $M$ be a core-free maximal subgroup of $G$. If $d > d(G)$, then we can find a generating set $g_1, \dots, g_d$ of $G$ with $d-1$ elements in $M$.
\end{prop}

\begin{proof}
    As previously, we will let $N := \soc(G) = S^n$ for some non-abelian simple group $S$, so that $G \leq \Aut(S) \wr S_n$. Let $\pi \colon G \to S_n$ be the natural projection map, and $\pi_i \colon S^n \to S_i$ the projection to each coordinate. 

    Since $d(G)< d$ we can find $x_1,\dots,x_{d-1}$ such that $G=\langle x_1,\dots,x_{d-1}\rangle.$ Since $M$ is core-free, $MN=G$, and so we can find $g_1, \dots, g_{d-1} \in M$ such that
    \[
        \langle g_1, \dots, g_{d-1} \rangle N = G.
    \]

    We will begin by assuming that $n \geq 2$.

    The permutation group $G \pi$ must be transitive, since $S^n$ is the minimal normal subgroup of $G$. Hence, there exists some $h \in \langle g_1, \dots, g_{d-1} \rangle$ such that $h \pi$ maps $2$ to $1$. Let $h = (\beta_1, \dots, \beta_n) (h \pi)$ for some $\beta_i \in \Aut(S)$. 

    By \cite[Theorem 1]{king}, there exists $a, b \in S$ such that $\langle a, b \rangle = S$, and $|a| = 2, |b| = p$ for some odd prime $p$. We will let
    \[
        g_d := (a, b^{\beta_2^{-1}}, 1, \dots, 1),
    \]
    and will consider
    \[
        H := \langle g_1, \dots, g_d \rangle.
    \]

    Since $h \in H$, we have $g_d^h = (b, \dots) \in H$. Hence, $(H \cap S^n)\pi_1$ contains $a, b$, and thus is equal to $S$. As $H \pi$ is transitive, this implies that $(H \cap S^n)\pi_i = S$ for any $i$. Additionally, $H \cap S_1$ contains $a^{|b|} \neq 1$ by our choice of $a,b$. Thus, by Goursat's Lemma, we must have $H \cap S^n = S^n$, concluding the proof.

    If $n=1$, then $G$ is almost simple. We may assume that $g_1 \neq 1$, as $M \cap S \neq 1$ by the end of the proof of the main theorem in \cite{lps}. Hence, by \cite[Theorem 1]{gk} there exists some $s \in S$ such that $\langle g_1, s \rangle \geq S$. Thus, letting $g_d := s$, we have
    \[
        \langle g_1, \dots, g_d \rangle = G,
    \]
    as required.
\end{proof}

\begin{corollary}
    If $M$ is a maximal subgroup of a finite group $G$ of type 2 such that $d(G) > d(G/M_G)$, then 
    \[
        \mathscr{D}_M(G) = d(G) - 1.
    \]
\end{corollary}

In order to consider whether it's possible to improve the bound in Proposition \ref{prop:type2maximalbound}, we must first list the possibilities for $M$. In the proceeding arguments, assume $G$ is a monolithic group, and $M \mleq G$ a core-free maximal subgroup of $G$. Let $S$ be the non-abelian simple group such that $S^n \leq G \leq \Aut(S) \wr S_n$, and $\pi \colon G \to S_n$ the natural projection map. By the maximality of $M$, we must have $M = (M \cap S^n) . \frac{G}{S^n}$. Additionally, since $S$ is a non-abelian simple group, we have the following possibilities:
\begin{enumerate}
    \item $M \cap S^n$ is conjugate to $(R \cap S)^n$, where $R \mleq K_1$, with $K_1$ defined as in the remark above Definition \ref{def:Htilde};
    \item $M \cap S^n = \prod_{B \in \Phi} D_B$ for a partition $\Phi$ of $\{1,\dots,n\}$, and $D_B$ a full diagonal subgroup on $\prod_{i \in B} S$;
    \item $M \cap S^n = 1$.
\end{enumerate}

If $M \cap S^n$ is conjugate to $(R \cap S)^n$, and $R < K_1$ is a flexible subgroup of $K_1$, then we can apply Lemma \ref{lemma:constructiveargument} as in Proposition \ref{prop:type2maximalbound} above to show that
\[
    \mathscr{D}_M(G) = d(G) - 1.
\]
We now state the following propositions in order to illustrate that there are many cases where $R$ is flexible, although this isn't always the case.

\begin{prop}
    Let $S = A_n$, with $n \geq 5$ and $n \neq 6$, and $S \leq K \leq \Aut(S) = S_n$. Then every maximal subgroup $R$ of $K$ not containing $S$ is flexible.
\end{prop}

\begin{proof}
    We can use the O'Nan-Scott classification to obtain a description of all possible $R$. It is then easy to check that each $A_n \nleq R \mleq K$ is flexible. For example, if $R$ is almost simple, with $\soc(R)=T,$ then, by \cite[Section 2]{lm}, there exists a prime $p$ with the property that for every $r\in T$ there is $t\in T$ such that $|r|_p\neq |rt|_p$. Since $T\cap  R\leq A_n$, this implies that $R$ is flexible
\end{proof}

\begin{prop}
    Let $S = A_6$. Then there exists a maximal subgroup $R$ of $\Aut(S)$, not containing $S$, which is not flexible.
\end{prop}

\begin{proof}
    Using Magma, we can find two maximal subgroups $R$ of $\Aut(S)$ which are not flexible. 
    
    The first is a maximal subgroup of size 40, where
    \[
        \big\{ \{|a| : a \in g(R \cap A_6)\} : g \in R \setminus (R \cap A_6) \big\} = \big\{ \{2,10\}, \, \{4\} \big\}.
    \]

    The second is a maximal subgroup of size $144$, where
    \[
        \big\{ \{|a| : a \in g(R \cap A_6)\} : g \in R \setminus (R \cap A_6) \big\} = \big\{ \{2,6\}, \, \{4\}, \, \{8\} \big\}.\qedhere
    \]
\end{proof}

However, even in the cases where $R$ is not flexible, it may still be possible to apply Corollary \ref{cor:alterg2} if we can carefully choose $g_1, g_2$ as part of a generating set to satisfy the conditions of this corollary.

On the other hand, if $M \cap S^n = \prod_{B \in \phi} D_B$ or $M \cap S^n = 1$, then it becomes much harder to apply a method such as that of Lemma \ref{lemma:constructiveargument}. Given elements $g_1, \dots, g_{d(G)} \in M$ with $\langle g_1, \dots, g_{d(G)} \rangle N = G$, we can replace $g_1$ with $v g_1$ for any $v \in M \cap S^n$. Previously, this allowed us quite a lot of freedom at this stage, especially when $R \mleq K_1$ is flexible. On the other hand, since now $M \cap S^n$ is small, possibly trivial, we have very little freedom at this stage. Hence, there's no guarantee that, for a given $g_1, \dots, g_d$, we can find $vg_1$ satisfying the conditions of Corollary \ref{cor:alterg2}. Hence, if we wish to apply such a method, we must instead consider all possibilities $g_1, \dots, g_{d(G)} \in M$ with $\langle g_1, \dots, g_{d(G)} \rangle N = G$, and find whether there exists $g_1$ satisfying the conditions of Corollary \ref{cor:alterg2}.


\begin{thebibliography}{99} 
\bibitem{ag} M. Aschbacher and R. Guralnick, 
On abelian quotients of primitive groups,
Proc. Amer. Math. Soc. 107 (1989), no. 1, 89–95.
\bibitem{bl} R. Baddeley and A. Lucchini, On representing finite lattices as intervals in subgroup lattices of finite groups, J. Algebra 196 (1997) 1-100.
\bibitem{classes} A. Ballester-Bolinches and L. Ezquerro, Classes of finite groups, Mathematics and Its Applications (Springer), 584. Springer, Dordrecht, 2006.
\bibitem{bgh} T. Burness, R. Guralnick and S. Harper, 
The spread of a finite group,
Ann. of Math. (2) 193 (2021), no. 2, 619–687.
\bibitem{cam} P. J Cameron, Aparna Lakshmanan S, Midhuna V Ajith, Hypergraphs defined on algebraic structures, Commun. Comb. Optim.  to appear.
\bibitem{pr} F. Dalla Volta and A. Lucchini, The smallest group with nonzero presentation rank, J. Group Theory 2 (1999), 147-155.
\bibitem{delu} E. Detomi and A. Lucchini,  Probabilistic generation of finite groups with a unique minimal normal subgroup, J. Lond. Math. Soc. (2) 87 (2013), no. 3, 689–706.
\bibitem{gk} R. Guralnik and W. Kantor, Probabilistic Generation of Finite Simple Groups, J. Algebra (2) 234 (2000) 743-792
\bibitem{king} C. King, Generation of finite simple groups by an involution and an element of prime order, J. Algebra 478 (2017) 153-173
\bibitem{lps} M. W. Liebeck, C. E. Praeger and J. Saxl, On the O’Nan-Scott theorem for finite primitive permutation groups,  J. Austral. Math. Soc. Ser. A 44  (1988), no. 3, 389–396.
\bibitem{alars} A. Lucchini, Generating hypergraphs of finite groups, Ars Combin. 161 (2024), 17–28.
\bibitem{lmsolvable} A. Lucchini and F. Menegazzo, Computing a Set of Generators of Minimal Cardinality in a Solvable Group, J. Symb. Comput. 17 (5) (1994) 409-420.
\bibitem{lm} A. Lucchini and F. Menegazzo, Generators for finite groups with a unique minimal normal subgroup, Rend. Sem. Mat. Univ. Padova 98 (1997), 173–191.
\bibitem{nongen} A. Lucchini and D. Nemmi, On the connectivity of the non-generating graph, Arch. Math. (Basel) 118 (2022), no. 6, 563–576.
\bibitem{q} M. Quick, Probabilistic generation of wreath products of non-abelian finite simple groups, Comm. Algebra 32 (2004), no. 12, 4753-4768.

\end{thebibliography}
\end{document}